\documentclass[12pt]{amsart}
\usepackage{amsmath}
\usepackage{amsthm}
\usepackage{amsfonts}
\usepackage{amssymb}
\usepackage{mathrsfs}
\usepackage{graphicx}
\usepackage{stmaryrd}
\usepackage{dsfont}
\usepackage[all]{xy}
\usepackage[shortlabels]{enumitem}
\usepackage{xcolor}
\usepackage[margin=1in]{geometry} 
\usepackage[bookmarks, bookmarksdepth=2, colorlinks=true, linkcolor=blue, citecolor=blue, urlcolor=blue]{hyperref}
\usepackage{eucal}
\usepackage{tikz,tikz-cd}

\newcommand{\arxiv}[1]{\href{http://arxiv.org/abs/#1}{{\tiny\tt arXiv:#1}}}
\newcommand{\DOI}[1]{\href{http://doi.org/#1}{\color{purple}{\tiny\tt DOI:#1}}}

\makeatletter
\@addtoreset{equation}{section}
\makeatother

\numberwithin{equation}{section}
\newtheorem{theorem}[equation]{Theorem}

\newtheorem{proposition}[equation]{Proposition}
\newtheorem{lemma}[equation]{Lemma}
\newtheorem{corollary}[equation]{Corollary}

\newtheorem{question}[equation]{Question}

\theoremstyle{definition}
\newtheorem{rmk}[equation]{Remark}
\newenvironment{remark}[1][]{\begin{rmk}[#1] \pushQED{\qed}}{\popQED \end{rmk}}
\newtheorem{eg}[equation]{Example}
\newenvironment{example}[1][]{\begin{eg}[#1] \pushQED{\qed}}{\popQED \end{eg}}
\newtheorem{defnaux}[equation]{Definition}
\newenvironment{definition}[1][]{\begin{defnaux}[#1]\pushQED{\qed}}{\popQED \end{defnaux}}

\newcommand{\bC}{\mathbf{C}}
\newcommand{\cC}{\mathcal{C}}
\newcommand{\fC}{\mathfrak{C}}

\newcommand{\cD}{\mathcal{D}}

\newcommand{\bF}{\mathbf{F}}

\newcommand{\cI}{\mathcal{I}}

\newcommand{\cM}{\mathcal{M}}

\newcommand{\bN}{\mathbf{N}}

\newcommand{\bQ}{\mathbf{Q}}

\newcommand{\fS}{\mathfrak{S}}

\newcommand{\defn}[1]{\textit{#1}}
\newcommand{\FI}{\mathbf{FI}}
\newcommand{\OI}{\mathbf{OI}}
\DeclareMathOperator{\init}{in}
\DeclareMathOperator{\Aut}{Aut}
\DeclareMathOperator{\im}{im}
\DeclareMathOperator{\Spec}{Spec}

\newcommand{\GL}{\mathbf{GL}}

\author{Robert P. Laudone}
\address{Department of Mathematics, University of Michigan, Ann Arbor, MI}
\email{\href{mailto:laudone@umich.edu}{laudone@umich.edu}}
\urladdr{\url{https://robertplaudone.github.io/}}
\thanks{RL was supported by NSF grant DMS-2001992}

\author{Andrew Snowden}
\address{Department of Mathematics, University of Michigan, Ann Arbor, MI}
\email{\href{mailto:asnowden@umich.edu}{asnowden@umich.edu}}
\urladdr{\url{http://www-personal.umich.edu/~asnowden/}}

\title[Systems of ideals parametrized by combinatorial structures]{Systems of ideals parametrized by\\ combinatorial structures}

\date{April 7, 2023}

\begin{document}

\begin{abstract}
A \defn{symmetric chain of ideals} is a rule that assigns to each finite set $S$ an ideal $I_S$ in the polynomial ring $\bC[x_i]_{i \in S}$ such that if $\phi \colon S \to T$ is an embedding of finite sets then the induced homomorphism $\phi_*$ maps $I_S$ into $I_T$. Cohen proved a fundamental noetherian result for such chains, which has seen intense interest in recent years due to a wide array of new applications. In this paper, we consider similar chains of ideals, but where finite sets are replaced by more complicated combinatorial objects, such as trees. We give a general criterion for a Cohen-like theorem, and give several specific examples where our criterion holds. We also prove similar results for certain limiting situations, where a permutation group acts on an infinite variable polynomial ring. This connects to topics in model theory, such as Fra\"iss\'e limits and oligomorphic groups.
\end{abstract}


\maketitle
\tableofcontents



\section{Introduction}

\subsection{Background}

Suppose that for each non-negative integer $n$ we have a Zariski closed subset $Z_n$ of $\bC^n$ satisfying the following condition: if $(x_1, \ldots, x_n)$ belongs to $Z_n$ and $i_1, \ldots, i_m$ are distinct indices, then $(x_{i_1}, \ldots, x_{i_m})$ belongs to $Z_m$. This means that whenever we sample $m$ distinct coordinates from a point in $Z_n$, the result belongs to $Z_m$. Taking $m=n$, we see that $Z_n$ is stable under the symmetric group $\fS_n$, which acts by permuting coordinates; we therefore refer to $Z_{\bullet}$ as a \defn{symmetric family} of varieties. See Example~\ref{ex:symmetric} for a typical example.

In the past 15 years, symmetric families\footnote{In many of these references, a slightly more general version of symmetric families is used; see \S \ref{ss:exsys}(c).} have found numerous applications: in algebraic statistics \cite{BD, HillarCampo, HillarSullivant, MarajNagel}, algebraic geometry \cite{DraismaNotes, DraismaEggermont, DraismaKuttler}, topology \cite{Ramos}, and even chemistry \cite{AschenbrennerHillar}. Additionally, there has been much work studying abstract properties of symmetric families \cite{KLS, LNNR, LNNR2, NagelRomer1, NagelRomer2, semilin, svar, sideals}; we note in particular that \cite{svar} classifies symmetric families of varieties. This is by no means a complete bibliography; see the surveys \cite{JKVLR} and \cite{DraismaNotes} for more details.

A \defn{symmetric chain} of ideals consists of an ideal $I_n$ of $\bC[x_1, \ldots, x_n]$ for each non-negative integer $n$ such that $I_n$ is $\fS_n$-stable, and $I_n \subset I_{n+1}$. If $Z_{\bullet}$ is a symmetric family of varieties then $\{I(Z_n)\}_{n \ge 0}$ is a symmetric chain of ideals. Cohen \cite{cohen} proved the following foundational result in this context\footnote{In fact, Cohen formulated his result in a slightly different language; see \S \ref{ss:intro-eq}. The formulation given here seems to be due to \cite{NagelRomer1}.}: if $I_{\bullet}$ is a symmetric chain of ideals, then there exists a non-negative integer $m$ and a finite collection of polynomials $f_1, \ldots, f_r \in \bC[x_1, \ldots, x_m]$ such that $I_n$ is generated by the $\fS_n$-orbits of $f_1, \ldots, f_r$ for all $n \ge m$. This finiteness property underlies all of the work referenced in the previous paragraph.

The theory of symmetric chains stems from the marriage of commutative algebra and the combinatorics of finite sets. In this paper, we wed commutative algebra with more complicated combinatorial structures. Our main theorem is an analog of Cohen's result in these more general settings.

\begin{example} \label{ex:symmetric}
Fix $r \ge 1$. Define $I_n \subset \bC[x_1, \ldots, x_n]$ to be the ideal generated by the discriminants of all $r+1$ element subsets of $\{x_1, \ldots, x_n\}$. Equivalently, for $n \ge r+1$, $I_n$ is the ideal generated by the $\fS_n$-orbit of
\begin{displaymath}
\Delta(x_1, \ldots, x_{r+1}) = \prod_{1 \le i < j \le r+1} (x_i-x_j),
\end{displaymath}
and $I_n=0$ for $n\le r$, by convention. Let $Z_n=V(I_n)$ be the associated variety. Then $I_{\bullet}$ is a symmetric chain of ideals, and $Z_{\bullet}$ is a symmetric family of varieties.

The variety $Z_n$ consists of all tuples $(a_1, \ldots, a_n)$ of complex numbers that contain at most $r$ distinct values. In other words, a point of $\bC^n$ belongs to $Z_n$ if whenever we sample $r+1$ distinct coordinates at least two are equal. Geometrically, $Z_n$ is a subspace arrangement: it is a union of linear subspaces of $\bC^n$ of dimension $r$. In particular, we see that the Krull dimension of $Z_n$ is not increasing with $n$, which is perhaps the first hint that a finiteness result like Cohen's might hold.
\end{example}

\subsection{Boron trees}

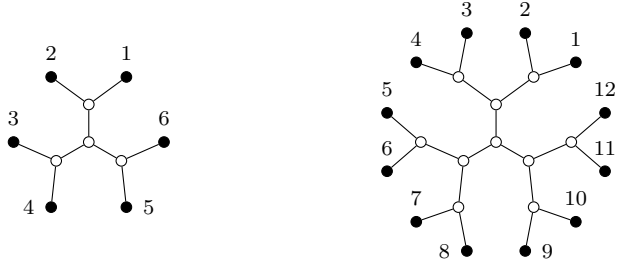
\begin{figure}
\begin{displaymath}
\begin{tikzpicture}[baseline=0]
\tikzstyle{boron}=[circle, draw, inner sep=0pt, minimum width=4pt]
\tikzstyle{vertex}=[circle, draw, fill=black, inner sep=0pt, minimum width=4pt]
\tikzstyle{edge}=[draw]
\node[boron] (a0) at (0,0) {};
\node[boron] (a1-1) at (90:.5) {};
\node[boron] (a1-2) at (210:.5) {};
\node[boron] (a1-3) at (330:.5) {};
\node[vertex,label={\tiny 1}] (a2-1) at (60:1) {};
\node[vertex,label={\tiny 2}] (a2-2) at (120:1) {};
\node[vertex,label={\tiny 3}] (a2-3) at (180:1) {};
\node[vertex,label=left:{\tiny 4}] (a2-4) at (240:1) {};
\node[vertex,label=right:{\tiny 5}] (a2-5) at (300:1) {};
\node[vertex,label={\tiny 6}] (a2-6) at (0:1) {};
\draw[edge] (a0) -- (a1-1);
\draw[edge] (a0) -- (a1-2);
\draw[edge] (a0) -- (a1-3);
\draw[edge] (a1-1) -- (a2-1);
\draw[edge] (a1-1) -- (a2-2);
\draw[edge] (a1-2) -- (a2-3);
\draw[edge] (a1-2) -- (a2-4);
\draw[edge] (a1-3) -- (a2-5);
\draw[edge] (a1-3) -- (a2-6);
\end{tikzpicture}
\hskip 1in
\begin{tikzpicture}[baseline=0]
\tikzstyle{boron}=[circle, draw, inner sep=0pt, minimum width=4pt]
\tikzstyle{vertex}=[circle, draw, fill=black, inner sep=0pt, minimum width=4pt]
\tikzstyle{edge}=[draw]
\node[boron] (a0) at (0,0) {};
\node[boron] (a1-1) at (90:.5) {};
\node[boron] (a1-2) at (210:.5) {};
\node[boron] (a1-3) at (330:.5) {};
\node[boron] (a2-1) at (60:1) {};
\node[boron] (a2-2) at (120:1) {};
\node[boron] (a2-3) at (180:1) {};
\node[boron] (a2-4) at (240:1) {};
\node[boron] (a2-5) at (300:1) {};
\node[boron] (a2-6) at (0:1) {};
\node[vertex,label={\tiny 1}] (a3-1) at (45:1.5) {};
\node[vertex,label={\tiny 2}] (a3-2) at (75:1.5) {};
\node[vertex,label={\tiny 3}] (a3-3) at (105:1.5) {};
\node[vertex,label={\tiny 4}] (a3-4) at (135:1.5) {};
\node[vertex,label={\tiny 5}] (a3-5) at (165:1.5) {};
\node[vertex,label={\tiny 6}] (a3-6) at (195:1.5) {};
\node[vertex,label={\tiny 7}] (a3-7) at (225:1.5) {};
\node[vertex,label=left:{\tiny 8}] (a3-8) at (255:1.5) {};
\node[vertex,label=right:{\tiny 9}] (a3-9) at (285:1.5) {};
\node[vertex,label={\tiny 10}] (a3-10) at (315:1.5) {};
\node[vertex,label={\tiny 11}] (a3-11) at (345:1.5) {};
\node[vertex,label={\tiny 12}] (a3-12) at (15:1.5) {};
\draw[edge] (a0) -- (a1-1);
\draw[edge] (a0) -- (a1-2);
\draw[edge] (a0) -- (a1-3);
\draw[edge] (a1-1) -- (a2-1);
\draw[edge] (a1-1) -- (a2-2);
\draw[edge] (a1-2) -- (a2-3);
\draw[edge] (a1-2) -- (a2-4);
\draw[edge] (a1-3) -- (a2-5);
\draw[edge] (a1-3) -- (a2-6);
\draw[edge] (a2-1) -- (a3-1);
\draw[edge] (a2-1) -- (a3-2);
\draw[edge] (a2-2) -- (a3-3);
\draw[edge] (a2-2) -- (a3-4);
\draw[edge] (a2-3) -- (a3-5);
\draw[edge] (a2-3) -- (a3-6);
\draw[edge] (a2-4) -- (a3-7);
\draw[edge] (a2-4) -- (a3-8);
\draw[edge] (a2-5) -- (a3-9);
\draw[edge] (a2-5) -- (a3-10);
\draw[edge] (a2-6) -- (a3-11);
\draw[edge] (a2-6) -- (a3-12);
\end{tikzpicture}
\end{displaymath}
\label{fig:boron}
\caption{Two boron trees. The boron atoms are white and the hydrogen atoms are black.}
\end{figure}

In a symmetric family of varieties, whenever one samples $m$ coordinates of a point in $Z_n$, the result belongs to $Z_m$. What happens if one constrains which $m$-tuples are allowed to be sampled? For example, one might only allow $m$-tuples that appear in order (see \S \ref{ss:exsys}(a)). We now discuss a more complicated variant, and state our results in this case.

A \defn{boron tree} is a tree $T$ in which the internal vertices (the ``boron atoms'') have valence three. See Figure~\ref{fig:boron} for an illustration. We write $T_H$ for the set of leaves (the ``hydrogen atoms''). There is an induced boron tree structure on any subset of $T_H$, and there is a notion of embeding for boron trees; see \S \ref{s:boron} for details.

Given a boron tree $T$, let $R_T$ be the polynomial ring $\bC[x_i]_{i \in T_H}$ with variables corresponding to the hydrogen atoms. If $\phi \colon T \to T'$ is an embedding of boron trees, there is an induced map $\phi_* \colon R_T \to R_{T'}$ given by $\phi_*(x_i)=x_{\phi(i)}$. We define a \defn{boric chain} to be a rule assigning to each boron tree $T$ an ideal $I_T$ in $R_T$ such that for any embedding $\phi \colon I_T \to I_{T'}$ we have $\phi_*(I_T) \subset I_{T'}$. This is the analog of a symmetric chain when the combinatorics of finite sets are replaced with the combinatorics of boron trees.

Our main theorem yields the following result in this particular case:

\begin{theorem} \label{thm:intro-boron}
Let $I_{\bullet}$ be a boric chain. Then there exist finitely many boron trees $T_1, \ldots, T_r$ and polynomials $f_i \in R_{T_i}$ such that for any boron tree $T$ the ideal $I_T$ is generated by the elements $\phi_*(f_i)$ as $i$ varies in $\{1,\ldots,r\}$ and $\phi$ varies over embeddings $T_i \to T$.
\end{theorem}

Boric chains are much more complicated than symmetric chains. The following example gives some indication of this.

\begin{example}
Let $T_0$ be the boron tree from Figure~\ref{fig:boron} with six leaves, and let $\delta \in R_{T_0}$ be the discriminant on the six variables. For a boron tree $T$, let $I_T$ be the ideal of $R_T$ generated by the elements $\phi_*(\delta)$, as $\phi$ varies over all embeddings $T_0 \to T$. Then $I_{\bullet}$ is a boric chain.

Consider the tree $T$ from Figure~\ref{fig:boron} with twelve leaves. The induced boron tree on $\{1,3,7,8,9,12\}$ is isomorphic to $T_0$, and so $I_T$ contains the discriminant
\begin{displaymath}
\Delta(x_1, x_3, x_7, x_8, x_9, x_{12}).
\end{displaymath}
The induced boron tree on $\{1,2,3,5,7,8\}$ is not isomorphic to $T_0$. Thus $I_T$ does not necessarily contain the discriminant
\begin{displaymath}
\Delta(x_1, x_2, x_3, x_5, x_7, x_8).
\end{displaymath}
In fact, it does not: the key point is that the discriminants on six variables are linearly independent in $R_T$ because they have distinct leading terms (if $i_1 < \cdots < i_6$ then the leading term of $\Delta(x_{i_1},\dots,x_{i_6})$ is $x_{i_1}^5 x_{i_2}^4\cdots x_{i_6}^0$).

Now let $T$ be an arbitrary boron tree, put $n=\# T_H$, and let $Z_T \subset \bC^n$ be the vanishing locus of the ideal $I_T$. This can be described as follows: a point of $\bC^n$ belongs to $Z_T$ if whenever we sample six distinct coordinates in ``relative position $T_0$,'' at least two of them are equal. Geometrically, $Z_T$ is a subspace arrangement in $\bC^n$, just as in Example~\ref{ex:symmetric}. However, the combinatorics of this subspace arrangement is more intricate than in the symmetric case. We also note that there are infinitely many boron trees $T$ into which $T_0$ does not embed, and thus $I_T=0$. In particular, the Krull dimension of $Z_T$ is unbounded as $T$ varies.
\end{example}

We do not prove anything about boric chains beyond Theorem~\ref{thm:intro-boron}. However, there are many intriguing directions to explore: for instance, one could study the projective dimension, regularity, or Hilbert series of boric chains, as the papers \cite{KLS, LNNR, LNNR2, NagelRomer1, NagelRomer2} do in the symmetric case; or one could attempt to classify boric chains, as the papers \cite{svar, sideals} do in the symmetric case. Theorem~\ref{thm:intro-boron} is almost certainly a pre-requisite for any further work on boric chains.

\subsection{Ideal systems}

We now describe our general results. Fix a field $k$. Let $\cC$ be a category equipped with a faithful conservative functor $\vert \cdot \vert \colon \cC \to \FI$, where $\FI$ is the category of finite sets and injections. Thus we can regard an object $A$ of $\cC$ as having an underlying set $\vert A \vert$. For an object $A$ of $\cC$, we let $R_A$ be the polynomial ring $k[x_i]_{i \in \vert A \vert}$. Given a morphism $\phi \colon A \to B$ in $\cC$, there is an induced ring homomorphism $\phi_* \colon R_A \to R_B$ by $\phi_*(x_i)=x_{\phi(i)}$. We define an \defn{ideal system} of $R$ to be a rule assigning to each object $A$ an ideal $I_A$ of $R_A$, such that for any morphism $\phi \colon A \to B$ we have $\phi_*(I_A) = I_B$. We say that $R$ is \defn{noetherian} if the ascending chain condition holds on ideal systems.

The following is a weak version of our main theorem:

\begin{theorem}
Suppose the following two conditions hold:
\begin{enumerate}
\item The category $\cC$ is orderable, i.e., it is possible to assign to each finite set $\vert A \vert$ a total order such that for each morphism $\phi \colon A \to B$ the function $\vert \phi \vert$ is order preserving.
\item The poset $\cM(\cC)$ consisting of isomorphism classes of weighted objects in $\cC$ is a well-quasi-order (see \S \ref{ss:monomial} for the definition).
\end{enumerate}
Then $R$ is noetherian.
\end{theorem}

The main point of the above theorem is that the hypotheses are purely combinatorial, while the conclusion is algebraic.

We say that $\cC$ is a \defn{G-category} if it satisfies conditions (a) and (b) of the theorem. We also define a broader class of categories, \defn{QG-categories}, where we only require an auxiliary category to be orderable. Our main theorem is that $R$ is noetherian whenever $\cC$ is a QG-category. This is an important improvement since many categories of interest are QG-categories but not G-categories; for instance, a G-category cannot have non-trivial automorphisms. This theorem, and its proof, are inspired by the approach of \cite{catgb}.

We show that the category $\FI$ of finite sets and injections is a QG-category; this recovers Cohen's theorem. We also show that the category of boron trees is a QG-category; this yields Theorem~\ref{thm:intro-boron}. We give a few additional examples as well.

\subsection{Equivariant ideals} \label{ss:intro-eq}

Suppose that $I_{\bullet}$ is a symmetric chain of ideals. Then $I=\bigcup_{n \ge 1} I_n$ is an ideal of the infinite variable polynomial ring $S=k[x_1, x_2, \ldots]$ that is stable under the action of the infinite symmetric group $\fS$. Cohen's theorem implies that $I$ (or any $\fS$-stable ideal of $S$) is generated by the $\fS$-orbits of finitely many elements. In other words, the ring $S$ is what is called \defn{$\fS$-noetherian}. In fact, this was Cohen's original formulation of his theorem.

In certain cases, our results can be recast in a similar manner. We explain the case of boron trees. The class of boron trees admits a so-called Fra\"iss\'e limit (see \S \ref{ss:fraisse}): this is a countably infinite boron tree $T$ that is extremely symmetrical, and into which every finite boron tree embeds. Let $\Gamma$ be the automorphism group of $T$, and let $\Omega=T_H$ be the set of hydrogen atoms. Using Theorem~\ref{thm:intro-boron}, we show that the polynomial ring $S=k[x_i]_{i \in \Omega}$ is $\Gamma$-noetherian.

More generally, one can pass to infinite permutation groups provided that $\cC$ is (the category associated to) a Fra\"iss\'e class. There are many interesting cases where this happens (such as with boron trees), though there are also many where it does not (such as with the permutation classes discussed in \S \ref{ss:perm}). The infinite permutation groups coming out of Fra\"iss\'e limits are the so-called \defn{oligomorphic groups}, and closely connected to model theory; see \cite{CameronBook} for general background.

\subsection{Related work} \label{ss:related}

This paper ties in with a number of threads of research.
\begin{itemize}
\item Of course, the primary antecedent of this paper is Cohen's work \cite{cohen, cohen2}; we have already discussed the profusion of related papers.
\item The Gr\"obner-theoretic methods of this paper are adapted from those of \cite{catgb}.
\item We were inspired to consider oligomorphic group actions on polynomial rings by the recent application of these groups in representation theory \cite{repst, homoten, bcat, Nekrasov}.
\item We study permutation group actions on infinite polynomial rings. Actions of linear groups, especially $\GL_{\infty}$, on infinite variable polynomial rings have also received much attention recently; see, e.g., \cite{polygeom, DraismaNoeth, sym2noeth, symc1, symu1}.
\item Polynomial rings with variables indexed by the vertices of a tree arise frequently in algebraic statistics; see, e.g., \cite{AD, GMN, SA}. We do not have a concrete application of our work here, but it feels potentially relevant, especially in light of the success of Cohen's original theorem in this area.
\end{itemize}

%
%

\subsection{Outline}

In \S \ref{s:sys}, we study systems of ideals in finite variable polynomial rings. In \S \ref{s:eq}, we study equivariant ideals in infinite variable polynomial rings. In \S \ref{s:boron} we examine boron trees in detail, which is one of the motivating examples.

%

\subsection*{Acknowledgments}

We thank Jan Draisma for helpful conversations.

\section{Systems of ideals} \label{s:sys}

\subsection{The basic objects}

The following definition introduces the categories that will organize our systems of polynomial rings. We recall that a functor $\Phi$ is \defn{conservative} if it is isomorphism reflecting, i.e., if $f$ is a morphism in the source category such that $\Phi(f)$ is an isomorphism then $f$ is an isomorphism.

\begin{definition}
An \defn{$\FI$-concrete category} is a category $\cC$ equipped with a faithful conservative functor $\vert \cdot \vert \colon \cC \to \FI$. If $\cC$ and $\cD$ are such categories, a \defn{concrete functor} $\Phi \colon \cC \to \cD$ is a functor equipped with a natural isomorphism $\vert A \vert \to \vert \Phi(A) \vert$.
\end{definition}

Let $\cC$ be an $\FI$-concrete category. For an object $A$ of $\cC$, we let $R_A=k[x_i]_{i \in \vert A \vert}$ be the polynomial ring over the field $k$ in variables indexed by the finite set $\vert A \vert$. If $\phi \colon A \to B$ is a morphism in $\cC$ then there is an induced $k$-algebra homomorphism $\phi_* \colon R_A \to R_B$ defined by $\phi_*(x_i)=x_{\phi(i)}$. Thus $R$ defines a functor from $\cC$ to $k$-algebras. When the dependence on $\cC$ needs to be indicated, we write $R(\cC)$ in place of $R$.

An \defn{ideal system} in $R$ is a rule attaching to each object $A$ of $\cC$ an ideal $I_A$ in the ring $R_A$ such that for any map $\phi \colon A \to B$ in $\cC$ we have $\phi_*(I_A) \subset I_B$. If $I$ and $J$ are two ideal systems, we write $I \subset J$ to mean $I_A \subset J_A$ for all objects $A$ of $\cC$. We say that $R$ is \defn{noetherian} if ideal systems of $R$ satisfy the ascending chain condition.

\begin{remark} \label{rmk:multi}
Suppose $(\cC, \vert \cdot \vert)$ is an $\FI$-concrete category. Fix $m \ge 1$, and let $[m]=\{1, \ldots, m\}$. For an object $A$ of $\cC$, put $\vert A \vert_m=\vert A \vert \times [m]$. Then $(\cC, \vert \cdot \vert_m)$ is also an $\FI$-concrete category. Let $R^m=R(\cC, \vert \cdot \vert_m)$; explicitly, $R^m_A=k[x_{i,j}]_{i \in \vert A \vert, j \in [m]}$. Thus our set-up accommodates multiple sets of variables simply by adjusting the functor $\vert \cdot \vert$.
\end{remark}

\subsection{Monomials} \label{ss:monomial}

Let $\cC$ be an $\FI$-concrete category. A \defn{weighting} on an object $A$ of $\cC$ is a function $\alpha \colon \vert A \vert \to \bN$. Such weightings correspond to monomials in the ring $R_A$. A \defn{weighted object} of $\cC$ is a pair $(A, \alpha)$ consisting of an object $A$ of $\cC$ and a weighting $\alpha$ on it. A \defn{morphism} of weighted objects $(A, \alpha) \to (B, \beta)$ is a morphism $\phi \colon A \to B$ in $\cC$ such that $\alpha(x) \le \beta(\phi(x))$ for all $x \in \vert A \vert$. In this way, we have a category $\cC^*$ of weighted objects. We note that $\phi \colon (A,\alpha) \to (B,\beta)$ is an isomorphism if and only if $\phi \colon A \to B$ is an isomorphism and $\alpha=\phi^*(\beta)$.

Let $\cM(\cC)$ be the set of isomorphism classes in $\cC^*$. We write $[A, \alpha]$ for the class of $(A, \alpha)$. Define a partial order on $\cM(\cC)$ by $[A,\alpha] \le [B,\beta]$ if there exists a map $(A,\alpha) \to (B,\beta)$ in $\cC^*$. One easily verifies that this is indeed a partial order; we note that the anti-symmetry of this order relies on the functor $\vert \cdot \vert$ being conservative.

Let $I$ be an ideal system of $R$. We say that $I$ is a \defn{monomial ideal system} if $I_A$ is a monomial ideal of $R_A$ for all objects $A$. For a weighted object $(A,\alpha)$, we let $m_{\alpha}$ be the monomial $\prod_{i \in \vert A \vert} x_i^{\alpha(i)}$ in the ring $R_A$. We define $\Phi(I) \subset \cM(\cC)$ to be the set of elements $[A,\alpha]$ such that the monomial $m_{\alpha}$ belongs to $I_A$.

\begin{proposition} \label{prop:monomial2}
The construction $\Phi$ defines an isomorphism of posets
\begin{displaymath}
\{ \text{monomial ideal systems of $R$} \} \stackrel{\sim}{\longrightarrow} \{ \text{order ideals of $\cM(\cC)$} \}.
\end{displaymath}
\end{proposition}

\begin{proof}
Let $I$ be an (arbitrary) ideal system of $R$. We first verify that $\Phi(I)$ is an order ideal in $\cM(\cC)$. Thus suppose $[A, \alpha] \in \Phi(I)$, and we have $[A,\alpha] \le [B,\beta]$. By definition of the partial order, there is a morphism $\phi \colon A \to B$ in $\cC$ such that $\alpha(x) \le \beta(\phi(x))$ for all $x \in \vert A \vert$. This implies that $\phi_*(m_{\alpha})$ divides $m_{\beta}$ in the ring $R_B$. Since $I_B$ contains $\phi_*(I_A)$, and thus $\phi_*(m_{\alpha})$, it follows that it contains $m_{\beta}$. Thus $[B,\beta] \in \Phi(I)$, as required.

Now suppose that we are given an order ideal $\cI$ of $\cM(\cC)$. For an object $A$ of $\cC$, define $I_A$ to be the $k$-span of the monomials $m_{\alpha}$ for which $[A,\alpha] \in \cI$. We claim that $I$ is a monomial ideal system of $R$. First, it is clear that $I_A$ is a monomial ideal of $R_A$: indeed, if $m_{\alpha} \in I_A$ and $m_{\alpha}$ divides some monomial $m_{\beta}$ of $R_A$ then $[A,\alpha] \le [A,\beta]$, so $[A,\beta]$ belongs to $\cI$, and so $m_{\beta}$ belongs to $I_A$. Now suppose $\phi \colon A \to B$ is a morphism in $\cC$. We must show $\phi_*(I_A) \subset I_B$. Let $[A,\alpha] \in \cI$ be given. Define a weighting $\beta$ on $B$ by $\beta(i)=\alpha(\phi^{-1}(i))$ if $i \in \im(\vert \phi \vert)$, and $\beta(i)=0$ otherwise. Then $\phi$ defines a morphism $(A,\alpha) \to (B,\beta)$ in $\cC^*$, and so $[A,\alpha] \le [B,\beta]$ in $\cM(\cC)$. Since $\cI$ is an order ideal, it follows that $[B,\beta] \in \cI$, and so $m_{\beta} \in I_B$. On the other hand, $m_{\beta}=\phi_*(m_{\alpha})$. Since $I_A$ is spanned by such $m_{\alpha}$'s, it follows that $\phi_*(I_A) \subset I_B$, which proves the claim. We put $\Psi(\cI)=I$.

If $I$ is a monomial ideal system of $R$ then $I=\Psi(\Phi(I))$: indeed, $\Psi(\Phi(I))_A$ is the $k$-span of the monomials contained in $I_A$, which is $I_A$ since $I_A$ is a monomial ideal. Similarly, if $\cI$ is an order ideal of $\cM(\cC)$ then $\cI=\Psi(\Phi(\cI))$. Indeed, it is clear that $\cI \subset \Psi(\Phi(\cI))$. We show the reverse. Thus suppose that $[A,\alpha] \in \Psi(\Phi(\cI))$. This means that the monomial $m_{\alpha}$ belongs to the ideal $\Phi(\cI)_A$. But since $\Phi(\cI)_A$ is spanned by a set of monomials, it follows that $m_{\alpha}$ is one of the monomials in this spanning set, and so $[A,\alpha] \in \cI$.

We thus see that $\Phi$ and $\Psi$ are mutually inverse bijections. It is clear that each is order preserving, and so the result follows.
\end{proof}

\begin{corollary} \label{cor:mono-acc}
Suppose that $\cM(\cC)$ is wqo. Then monomial ideal systems of $R$ satisfy ACC.
\end{corollary}

Recall that wqo stands for well-quasi-order. We refer to \cite[\S 1.2]{BV} for background on this concept.

\subsection{Orders}

Let $\cC$ be an $\FI$-concrete category. An \defn{ordering} on $\cC$ is a rule assigning to each object $A$ of $\cC$ a total order on the set $\vert A \vert$ such that if $\phi \colon A \to B$ is a morphism in $\cC$ then $\vert \phi \vert \colon \vert A \vert \to \vert B \vert$ is order-preserving. In other words, an ordering on $\cC$ is a lift of the functor $\vert \cdot \vert \colon \cC \to \FI$ through the forgetful functor $\OI \to \FI$, where $\OI$ is the category of totally ordered finite sets and monotonic injections. We say that $\cC$ is \defn{orderable} if it admits an ordering.

Fix an ordering on $\cC$. For an object $A$ of $\cC$, we thus have a total order on the variables in $R_A$, which induces a monomial order on $R_A$ (lexicographic order). For an element $f \in R_A$ we let $\init(f)$ be its initial term, and for an ideal $I \subset R_A$ we let $\init(I)$ be its initial ideal. For an ideal system $I$ of $R$, we define $\init(I)$ to be the rule assigning to an object $A$ the the ideal $\init(I_A)$ of $R_A$.

\begin{proposition} \label{prop:init}
If $I$ is an ideal system of $R$ then $\init(I)$ is a monomial ideal system of $R$.
\end{proposition}

\begin{proof}
By definition, $\init(I_A)$ is a monomial ideal. Let $\phi \colon A \to B$ be a morphism in $\cC$ and let $m \in \init(I_A)$ be a monomial. Then $m=\init(f)$ for some $f \in I_A$. Since $\vert \phi \vert \colon \vert A \vert \to \vert B \vert$ is order-preserving, the map $\phi_*$ commutes with formation of initial terms, and so $\phi_*(m)=\init(\phi_*(f))$. Since $\phi_*(f)$ belongs to $I_B$, it follows that $\phi_*(m)$ belongs to $\init(I_B)$. We thus see that $\phi_*(\init(I_A)) \subset \init(I_B)$, which completes the proof.
\end{proof}

\begin{proposition} \label{prop:grobner}
Let $I \subset J$ be ideal systems of $R$ such that $\init(I)=\init(J)$. Then $I=J$.
\end{proposition}

\begin{proof}
This is the usual Gr\"obner lemma.
\end{proof}

\subsection{Transfering noetherianity}

The next proposition shows how we can sometimes transfer noetherianity results between different categories.


\begin{proposition} \label{prop:F}
Let $\Phi \colon \cD \to \cC$ be a concrete functor of $\FI$-concrete categories that is essentially surjective. If $R(\cD)$ is noetherian then $R(\cC)$ is also noetherian.
\end{proposition}

\begin{proof}
Put $R=R(\cD)$ and $R'=R(\cC)$, and let $I'$ be an ideal system of $R'$. If $A$ is an object of $\cC$ then we have a natural isomorphism $\vert A \vert = \vert \Phi(A) \vert$ (since $\Phi$ is a concrete functor), and thus a natural identification $R_A=R'_{\Phi(A)}$. We define $I_A$ to be the ideal of $R_A$ corresponding to $I'_{\Phi(A)} \subset R'_{\Phi(A)}$. If $\phi \colon A \to B$ is a morphism in $\cD$ then $\phi_* \colon R_A \to R_B$ is identified with $\Phi(\phi)_* \colon R'_{\Phi(A)} \to R'_{\Phi(B)}$. Since $I'$ is stable under the latter, it follows that $I$ is stable under the former. Thus $I$ is an ideal system of $R$.

Now suppose that $I'_{\bullet}$ is an ascending chain of ideal systems in $R'$. We thus obtain an ascending chain $I_{\bullet}$ of ideal systems of $R$. Since $R$ is noetherian, this chain stabilizes, and so there is some $n$ such that $I_n=I_m$ for any $m \ge n$. We thus see that $(I'_n)_{\Phi(A)}=(I'_m)_{\Phi(A)}$ for any $m \ge n$ and any object $A$ of $\cD$. Since $\Phi$ is essentially surjective, it follows that $I'_n=I'_m$, and so the original chain stabilizes. Thus $R'$ is noetherian.
\end{proof}

\subsection{The main theorem}

The following definition and theorem give a purely combinatorial criterion for the system $R$ of polynomial rings to be noetherian.

\begin{definition}
A \defn{G-category} is an $\FI$-concrete category that is orderable and for which $\cM(\cC)$ is wqo. A \defn{QG-category} is an $\FI$-concrete category for which there exists a G-category $\cD$ and a concrete functor $\cD \to \cC$ that is essentially surjective.
\end{definition}

\begin{theorem} \label{thm:noethsys}
Let $\cC$ be a QG-category. Then $R$ is noetherian.
\end{theorem}

\begin{proof}
First suppose $\cC$ is a G-category, and fix an ordering on $\cC$. Let $I_{\bullet}$ be an ascending chain of ideal systems in $R$. Then $\init(I_{\bullet})$ is an ascending chain of monomial ideal systems in $R$ (Proposition~\ref{prop:init}), and thus stabilizes (Corollary~\ref{cor:mono-acc}). Hence the original chain stabilizes (Proposition~\ref{prop:grobner}), and so $R$ is noetherian.

We now treat the general case. Let $\cD$ be a G-category and let $\Phi \colon \cD \to \cC$ be a concrete functor that is essentially surjective. By the previous paragraph, $R(\cD)$ is noetherian. Thus $R(\cC)$ is noetherian (Proposition~\ref{prop:F}), which completes the proof.
\end{proof}

\begin{remark}
The G-category and QG-category terminology alludes to the Gr\"obner and quasi-Gr\"obner terminology from \cite{catgb}, though the definitions differ. For example, consider the category {\bf BOI} (``block $\OI$''), with objects given by finite totally ordered sets and morphisms given by block embeddings. That is for $n \leq m$, ${\rm Hom}_{\bf BOI}([n],[m]) = \{\phi_0,\dots,\phi_{m-n}\}$ where $\phi_i(j) = j+i$. One can show that this is a Gr\"obner category in the sense of \cite{catgb}, but not a G-category (or even a QG-category).
\end{remark}

\begin{remark}
Let $\cC$ be an $\FI$-concrete category, and let $R=R(\cC)$. An \defn{$R$-module system} is a rule assigning to each object $A$ of $\cC$ an $R_A$-module $M_A$ and to each morphism $\phi \colon A \to B$ in $\cC$ an additive map $\phi_* \colon M_A \to M_B$ satisfying $\phi_*(am)=\phi_*(a) \phi_*(m)$, where $a \in R_A$ and $m \in M_A$. There is a notion of finite generation and noetherianity for $R$-module systems. By combining the methods here with those of \cite{catgb}, one can give a combinatorial criterion that ensures all finitely generated $R$-module systems are noetherian. In the case $\cC = {\bf FI}$, $R(\cC)$-module systems were investigated in \cite{NagelRomer2}.
\end{remark}

\subsection{Examples} \label{ss:exsys}

We give a few simple examples of Theorem~\ref{thm:noethsys} here.
\begin{enumerate}
\item Let $\cC=\OI$, and take $\vert \cdot \vert$ to be the identity functor. This is an $\FI$-concrete category with a natural ordering. The set $\cM(\cC)$ consists of all words in the alphabet $\bN$, endowed with the Higman order; precisely, $n_1 \cdots n_r \le m_1 \cdots m_s$ if there is an injection $\phi \colon [r] \to [s]$ such that $n_i \le  m_{\phi(i)}$ for all $i \in [r]$. The set $\cM(\cC)$ is wqo by Higman's lemma. Thus $\cC$ is a G-category, and so $R$ is noetherian.
\item More generally, we can consider $\cC=\OI$ equipped with the functor $\vert \cdot \vert_m$ from Remark~\ref{rmk:multi}. Here $\cM(\cC)$ consists of all words in the alphabet $\bN^m$. Since $\bN^m$ is wqo (by Dickson's lemma), Higman's lemma again shows that $\cM(\cC)$ is a wqo. Thus $\cC$ is a G-category, and so $R$ is noetherian.
\item Let $\cC=\FI$, with its natural $\FI$-concrete structure. This category is not orderable. However, letting $\cD=\OI$, we have a concrete functor $\cD \to \cC$, namely, the forgetful functor. It follows that $\cC$ is a QG-category, and so $R$ is noetherian. This recovers Cohen's original theorem \cite{cohen}, as formulated in \cite{NagelRomer2}. There is a similar example using $\vert \cdot \vert_m$, which recovers Cohen's more general theorem from \cite{cohen2}.
\end{enumerate}

\subsection{Another example} \label{ss:perm}

We now give a somewhat more substantial example. A \defn{permutation} is a finite set equipped with two total orders, and an \defn{embedding} of permutations is an injection of finite sets that is monotonic with respect to both orders. A \defn{permutation class} is a class $\fC$ of permutations such that whenever $X$ belongs to $\fC$, any permutation embedding into $X$ also belongs to $\fC$. See \cite[\S 1.1]{BV} for additional background.

Let $\fC$ be a permutation class. Define $\cC$ to be the category whose objects are the members of $\fC$, and whose morphisms are embeddings. For a permutation $X$, we let $\vert X \vert$ be its underyling set; this endows $\cC$ with an $\FI$-concrete structure. The category $\cC$ is clearly orderable: just use the first total order on the permutations. Thus, if $\cM(\cC)$ is wqo, then $\cC$ is a G-category and the ring $R(\cC)$ is noetherian.

In \cite[\S 1.8]{BV}, the notion of \defn{labeled well-quasi-order (lwqo)} is discussed in the context of permutation classes. It is clear from the definition that if $\fC$ is lwqo then $\cM(\cC)$ is wqo. Thus the theory of \cite{BV} provides many examples of permutation classes for which $R(\cC)$ is noetherian. For example, this is the case if $\fC$ is the class of separable permutations; see \cite[Corollary~7.6]{BV} and the discussion following \cite[Corollary~7.7]{BV}.

\section{Equivariant ideals} \label{s:eq}

\subsection{The main theorem} \label{ss:equimain}

Let $\Gamma$ be a permutation group with domain $\Omega$. Let $S=k[x_i]_{i \in \Omega}$ be the polynomial ring equipped with its action of $\Gamma$. A \defn{$\Gamma$-ideal} of $S$ is an ideal stable by $\Gamma$. We say that $S$ is \defn{$\Gamma$-noetherian} if the ascending chain condition holds for $\Gamma$-ideals.

Define a category $\cC=\cC(\Gamma,\Omega)$ as follows. The objects are finite subsets of $\Omega$. A morphism $A \to B$ is a function $\phi \colon A \to B$ for which there exists $g \in \Gamma$ such that $\phi(x)=gx$ for all $x \in A$. The forgetful functor $\cC \to \FI$ makes $\cC$ into an $\FI$-concrete category. Let $R=R(\cC)$ be the system of polynomial rings associated to $\cC$.

\begin{theorem} \label{thm:RtoS}
If $R$ is noetherian then $S$ is $\Gamma$-noetherian.
\end{theorem}

\begin{proof}
Let $J$ be a $\Gamma$-ideal of $S$. For a finite subset $A$ of $\Omega$, let $I_A$ be the contraction of $J$ to $R_A \subset S$. Then $I$ is a system of ideals in $R$. Moreover, $J=\bigcup_A I_A$, and so one can recover $J$ from $I$. It follows that the function
\begin{displaymath}
\{ \text{$\Gamma$-ideals of $S$} \} \to \{ \text{ideal systems of $R$} \}
\end{displaymath}
is injective. It is also clearly order-preserving. Since the target above satisfies ACC, so does the source.
\end{proof}

It is not difficult to see if $\cC$ is a G-category directly from $(\Gamma, \Omega)$. Let $\Omega^*$ be the set of all functions $\Omega \to \bN$ of finite support, and let $\cM=\cM(\Gamma, \Omega)$ be the set of $\Gamma$-orbits on $\Omega^*$. For $\alpha \in \Omega^*$ we write $[\alpha]$ for its class in $\cM$. We partially order $\cM$ by $[\alpha] \le [\beta]$ if there exists $g \in \Gamma$ such that $\alpha(x) \le \beta(gx)$ for all $x \in \Omega$.

\begin{proposition}
Suppose $\cM(\Gamma,\Omega)$ is wqo and $\Gamma$ preserves a total order on $\Omega$. Then $\cC$ is a G-category, and $S$ is $\Gamma$-noetherian
\end{proposition}

\begin{proof}
One easily sees that $\cM(\cC)$ is isomorphic to $\cM(\Gamma,\Omega)$, and thus $\cM(\cC)$ is wqo. Fix an ordering $<$ on $\Omega$ preserved by $\Gamma$. This restricts to an ordering on each finite subset of $\Omega$, and thus induces an ordering on $\cC$. Thus $\cC$ is a G-category. Theorem~\ref{thm:noethsys} thus implies that $R$ is noetherian, and so $S$ is $\Gamma$-noetherian from Theorem~\ref{thm:RtoS}.
\end{proof}

\subsection{Examples} \label{ss:S-example}

We give a few simple examples of Theorem~\ref{thm:RtoS}.
\begin{enumerate}
\item Consider the infinite symmetric group $\fS$ acting on $\Omega=\{1,2,\ldots\}$. The associated category $\cC$ is then equivalent to $\FI$. The ring $R$ is noetherian by \S \ref{ss:exsys}(c), and so the ring $S=k[x_i]_{i \in \Omega}$ is $\fS$-noetherian by Theorem~\ref{thm:RtoS}. This recovers Cohen's theorem from \cite{cohen}, in its original form.
\item We can also recover Cohen's theorem for $m$ sets of variables \cite{cohen2} in its original form by considering the action of $\fS$ on $\Omega^m$. Then $\cC$ is the $\FI$-concrete category $(\FI, \vert \cdot \vert_m)$, where $\vert \cdot \vert$ denotes the identity functor $\FI \to \FI$. The ring $R$ is again noetherian, and so $S=k[x_{i,j}]_{i \in \Omega, j \in [m]}$ is $\fS$-noetherian.
\item Let $\Omega=\bQ$ be the set of rational numbers, and let $\Gamma$ be the group of order preserving bijections $\Omega \to \Omega$. One easily sees that $\cC$ is equivalent to $\OI$ in this case. Thus $R$ is noetherian by \S \ref{ss:exsys}(a), and so $S=k[x_i]_{i \in \bQ}$ is $\Gamma$-noetherian by Theorem~\ref{thm:RtoS}. It is not difficult to deduce this directly from Cohen's theorem, but as far as we know, this case had not been remarked on previously.
\end{enumerate}

\subsection{Necessary conditions} \label{ss:necc}

Let $(\Gamma, \Omega)$ be a permutation group. Let $\cM^{\circ}=\cM^{\circ}(\Gamma, \Omega)$ be the set of $\Gamma$-orbits of finite subsets of $\Omega$. We partially order $\cM^{\circ}$ by $[A] \le [B]$ if $A \subset gB$ for some $g \in \Gamma$. Note that $\cM^{\circ}$ is the subset of $\cM$ consisting of classes $[\alpha]$ where the function $\alpha \in \Omega^*$ is valued in $\{0,1\}$.

Let $S=k[x_i]_{i \in \Omega}$, as usual. We say that $S$ is \defn{topologically $\Gamma$-noetherian} if radical $\Gamma$-stable ideals satisfy ACC; equivalently, $\Gamma$-stable closed subsets of $\Spec(S)$ satisfy DCC. If $S$ is $\Gamma$-noetherian then it is clearly topologically $\Gamma$-noetherian. The converse does not hold in general (see the discussion of Krasilnikov's example before \cite[Proposition~3.5]{DraismaNotes}).

\begin{proposition} \label{prop:necc}
Suppose that $S=k[x_i]_{i \in \Omega}$ is topologically $\Gamma$-noetherian. Then $\cM^{\circ}$ is wqo.
\end{proposition}

\begin{proof}
Adapting the proof of Proposition~\ref{prop:monomial2}, one can show that there is an isomorphism of posets
\begin{displaymath}
\{ \text{$\Gamma$-stable squarefree monomial ideals of $S$} \} \stackrel{\sim}{\longrightarrow} \{ \text{order ideals of $\cM^{\circ}$} \}.
\end{displaymath}
Here, a \defn{squarefree monomial ideal} is an ideal generated by squarefree monomials. One easily sees that squarefree monomial ideals are radical. Since $S$ is topologically $\Gamma$-noetherian, the left side above therefore satisfies ACC. Hence the right side does as well.
\end{proof}

The finiteness condition on $\cM^{\circ}$ appearing above implies the oligomorphic condition:

\begin{proposition}
Suppose that $\cM^{\circ}$ is wqo. Then $\Gamma$ is an oligomorphic permutation group, that is, $\Gamma$ has finitely many orbits on $\Omega^n$ for all $n$.
\end{proposition}

\begin{proof}
Suppose $\Gamma$ is not oligomorphic. Let $\Omega^{(n)}$ denote the set of $n$-element subsets of $\Omega$. One easily sees that $\Gamma$ has infinitely many orbits on $\Omega^{(n)}$ for large enough $n$. Suppose that $A_1, A_2, \ldots$ are $n$-element subsets of $\Omega$ belonging to distinct orbits. Then $[A_i]$ and $[A_j]$ are incomparable in $\cM^{\circ}$ for all $i \ne j$. Indeed, if $[A_i] \le [A_j]$ then we would have $gA_i \subset A_j$ for some $g \in \Gamma$, and thus $gA_i=A_j$ since $A_i$ and $A_j$ have the same cardinality, contradicting that $A_i$ and $A_j$ belong to distinct orbits. We thus see that the $[A_i]$'s give an infinite antichain in $\cM^{\circ}$, and so $\cM^{\circ}$ is not wqo.
\end{proof}

The above propositions provide necessary conditions for $S$ to be $\Gamma$-noetherian, or even topologically $\Gamma$-noetherian: the action of $\Gamma$ must be oligomorphic, and $\cM^{\circ}$ must be wqo. The oligomorphic condition has been widely studied, especially in connection to model theory (see \cite{CameronBook, Macpherson}), and the wqo condition on $\cM^{\circ}$ has also been studied (see \cite{OudrarPouzet}). It is possible to prove similar necessary results in the setting from \S \ref{s:sys}. We now give some examples and further comments related to this.

\begin{example}
Consider the infinite symmetric group $\fS$ acting on $\Omega=\{1,2,\ldots\}$. The action of $\fS$ on $\Omega \times \Omega$ is oligomorphic. However, $\cM^{\circ}=\cM^{\circ}(\fS, \Omega \times \Omega)$ is not wqo. Indeed, define
\begin{displaymath}
A_n = \{ (1,2), (2,3), \ldots, (n-1,n), (n, 1) \}.
\end{displaymath}
This is an $n$-element subset of $\Omega \times \Omega$, and thus defines a class $[A_n]$ in $\cM^{\circ}$. One easily sees that these elements form an anti-chain, and so $\cM^{\circ}$ is not wqo. Thus, by Proposition~\ref{prop:necc}, we see that $S=k[x_{i,j}]_{i,j \ge 1}$ is not topologically $\fS$-noetherian. This is well-known; e.g., see \cite[Example~2.4]{DraismaNotes}.
\end{example}

\begin{example}
Let $\bF$ be a finite field, let $\Gamma=\bigcup_{n \ge 1} \GL_n(\bF)$ be the infinite general linear group, and let $\Omega=\bigcup_{n \ge 1} \bF^n$. As above, the action of $\Gamma$ on $\Omega$ is oligomorphic, but $\cM^{\circ}=\cM^{\circ}(\Gamma, \Omega)$ is not wqo. Indeed, let
\begin{displaymath}
A_n = \{e_1,e_2,\ldots,e_n,e_1+\cdots+e_n \},
\end{displaymath}
where $\{e_i\}$ is the standard basis of $\Omega$. Then $A_n$ is an $n+1$ element subset of $\Omega$ that is linearly dependent, but for which all proper subsets are linearly independent. It follows that $[A_n]$ is an antichain in $\cM^{\circ}$, and so $\cM^{\circ}$ is not wqo.  Thus, by Proposition~\ref{prop:necc}, we see that $S=k[x_v]_{v \in \Omega}$ is not topologically $\fS$-noetherian. We thank Jan Draisma for explaining this example to us.
\end{example}

\begin{remark}
Let $(\Gamma, \Omega)$ be an oligomorphic group. For a non-negative integer $n$, let $\phi(n)$ be the number of $\Gamma$-orbits on the set $\Omega^{(n)}$ of $n$-element subsets of $\Omega$. The function $\phi \colon \bN \to \bN$ is called the \defn{profile} of $(\Gamma, \Omega)$. In \cite[\S 6.1]{OudrarPouzet} the following question is posed: if $\cM^{\circ}$ is wqo, is it necessarily true that $\phi(n) \le c^n$ for some positive number $c$? As far as we know, this question is still open. A positive answer would imply that $S$ can only be $\Gamma$-noetherian when $\phi$ has exponential growth. We note that the profile functions in the previous two examples have super-exponential growth.
\end{remark}

\begin{question}
If $\cM$ is wqo then of course so is $\cM^{\circ}$, as the latter is a subset of the former. Does the converse hold? We know of no counterexample.
\end{question}

\subsection{Fra\"iss\'e theory} \label{ss:fraisse}

We have seen that $S$ can only be $\Gamma$-noetherian when $\Gamma$ is an oligomorphic permutation group. We now recall that primary way of constructing such groups, and then give some examples. We refer to \cite{CameronBook} for additional background.

A \defn{signature} $\Sigma$ is an alphabet $\{R_i\}_{i \in I}$ where each $R_i$ is assigned a positive integer $n_i$ called it arity. A \defn{structure} is a set $X$ equipped with an $n_i$-arity relation $R_i$ on $X$ (i.e., a subset of $X^{n_i}$) for each $i \in I$. There is a natural notion of embedding and isomorphism of structures. We are typically interested in classes of structures for some fixed signature; for example, if the signature consists of a single binary relation, we might consider partially ordered sets, equivalence relations, graphs, etc.

Fix $\Sigma$, and let $\fC$ be a class of finite structures. We say that $\fC$ is a \defn{Fra\"iss\'e class} if it satisfies the following conditions:
\begin{itemize}
\item $\fC$ is non-empty and has countably cardinality.
\item If $X$ belongs to $\fC$ and $Y$ embeds into $X$ then $Y$ belongs to $\fC$.
\item Suppose we have embeddings $i_1 \colon Y \to X_1$ and $i_2 \colon Y \to X_2$ of structures in $\fC$. Then there exists a structure $Z$ in $\fC$ with embedings $j_1 \colon X_1 \to Z$ and $j_2 \colon X_2 \to Z$ such that $j_1 \circ i_1 = j_2 \circ i_2$ and $Z=\im(j_1) \cup \im(j_2)$.
\end{itemize}
The third condition is called the \defn{amalgamation property} for $\fC$, and $Z$ is called an \defn{amalgamation} of $X_1$ and $X_2$ over $Y$. Note that the second condition implies that the empty structure $\emptyset$ belongs to $\fC$. If $X_1$ and $X_2$ are arbitrary structures in $\fC$ then, taking $Y=\emptyset$ in the amalgamation property, we find that there is a structure $Z$ in $\fC$ into which both $X_1$ and $X_2$ embed. This is called the \defn{joint embedding property} and sometimes stated separately from the amalgamation property.

A structure $\Omega$ is called \defn{homogeneous} if whenever $i$ and $j$ are embeddings of a finite structure $X$ into $\Omega$, there is an automorphism $\sigma$ of $\Omega$ such that $j=\sigma \circ i$. Fra\"iss\'e proved that if $\fC$ is a Fra\"iss\'e class then there is a countable homogeneous structure $\Omega$ such that $\fC$ is exactly the \defn{age} of $\Omega$, i.e., the class of finite structures that embed into $\Omega$. Such an $\Omega$ is called the \defn{Fra\"iss\'e limit} of $\Omega$, and is unique up to isomorphism. Let $\Gamma=\Aut(\Omega)$. By homogeneity, the $\Gamma$-orbits on the set $\Omega^{(n)}$ of $n$-element subsets of $\Omega$ are in bijection with isomorphism classes of $n$-element structures in $\fC$. Thus, if there are finitely many such structures (which is automatic if the signature $\Sigma$ is finite) then $\Gamma$ acts oligomorphically on $\Omega$. In this way, Fra\"iss\'e limits are the primary way of constructing oligomorphic groups.

We note that the category $\cC$ we associated to $(\Gamma, \Omega)$ in \S \ref{ss:equimain} is equivalent to the category whose objects are the structures in $\fC$, and whose morphisms are embedings of structures. The poset $\cM^{\circ}$ is identified with the set of isomorphism classes in $\fC$, ordered by embedability.

\begin{example}
Let $\fC$ be the class of all finite simple graphs; here the signature consists of a single binary relation which records if an edge is present between two vertices. This is a Fra\"iss\'e class. The limit $R$ is the well-known \defn{Rado graph} or \defn{random graph}. Its automorphism group $\Gamma$ acts oligomorphically on the vertex set $V(R)$ or $R$. The poset $\cM^{\circ}$ is not a wqo: one obtained an anti-chain by considering $n$-cycles for $n \ge 1$. Thus $S=k[x_i]_{i \in V(R)}$ is not $\Gamma$-noetherian, or even topologically $G$-noetherian.
\end{example}

\begin{example}
Fix a positive integer $c$. Let $\fC$ be the class of finite totally ordered sets equipped with a $c$-coloring; here the signature has one binary relation to record the order, and $c$ unary relations to record the coloring. This is a Fra\"iss\'e class. The limit can be described as the totally ordered set $\bQ$ equipped with a $c$-coloring $\sigma \colon \bQ \to [c]$ such that each color is dense; the uniqueness of the Fra\"iss\'e limit shows that such a structure is independent of the choice of coloring, up to isomorphism.

Let $\Gamma$ be the automorphism group of $(\bQ, <, \sigma)$, which acts oligomorphically on $\bQ$. The category $\cC$ is easily seen to be a $G$-category: the poset $\cM$ is the set of words in the alphabet $\bN \times [c]$, under the Higman order. Thus $R$ is noetherian by Theorem~\ref{thm:noethsys}, and so $S=k[x_i]_{i \in \bQ}$ is $\Gamma$-noetherian by Theorem~\ref{thm:RtoS}. We note that in this case, $\cM^{\circ}$ is the set of words in the finite alphabet $[c]$ with the Higman order. This oligomorphic group was recently studied in \cite{colored}.
\end{example}

\begin{example}
Recall the notion of permutation class discussed in \S \ref{ss:perm}. A theorem of Cameron \cite{CameronPerm} asserts that there are just five permutation classes that are Fra\"iss\'e classes. Thus most of the examples coming from \S \ref{ss:perm} will not have an associated oligomorphic group. For instance, there is no such group associated to the class of separable permutations.
\end{example}

\section{Boron trees} \label{s:boron}

A \defn{boron tree} is a finite tree whose internal vertices have valence three; see Figure~\ref{fig:boron}. The internal vertices are called ``boron atoms'' and the leaves ``hydrogen atoms.'' Let $T$ be a boron tree. We obtain a quaternary relation $\rho$ on the set $T_H$ of hydrogen atoms by declaring $\rho(w,x;y,z)$ to be true if the geodesic joining $(w,x)$ intersects the geodesic joining $(y,z)$. It is really the structure $(T_H, \rho)$ that we can about, though the tree can be recovered from it.

We let $\cC$ be the category whose objects are boron trees and whose morphisms are embeddings, meaning maps of hydrogen atoms preserving the quaternary relation. This is an $\FI$-concrete category via $\vert T \vert=T_H$. We let $\fC$ be the collection of all finite structures $(T_H, R)$ coming from boron trees. This is a Fra\"iss\'e class \cite[Exercise 2.6.4]{CameronBook}. We let $\Omega$ be the Fra\"iss\'e limit, and we let $\Gamma$ be its automorphism group. The following is our main theorem in this case:

\begin{theorem} \label{thm:boron}
The system $R=R(\cC)$ of finite polynomial rings is noetherian, and the polynomial ring $S=k[x_i]_{i \in \Omega}$ is $\Gamma$-noetherian.
\end{theorem}

We prove Theorem~\ref{thm:boron} by showing that $\cC$ is QG-category. Since boron trees can have non-trivial automorphisms, $\cC$ cannot be a G-category. We must therefore introduce an auxiliary category where the objects have more structure, to kill the automorphisms.

A \defn{planar boron tree} is one that is embedded in the unit circle, meaning the hydrogen atoms are placed on the circle and the boron atoms and edges are internal to the circle, and there are no crossings. The hydrogen atoms of a planar boron tree carry a natural cyclic order, where $y$ is said to be between $x$ and $z$ if one meets $y$ when traveling from $x$ to $z$ counterclockwise (and $x$, $y$, and $z$ are distinct). An \defn{ordered boron tree} is a planar boron tree with a distinguished hydrogen atom, labeled $\infty$ and called the root. In such a tree, the non-root hydrogen atoms are totally ordered, with $x<y$ if $x$ is between $\infty$ and $y$; we also define $x<\infty$ for all $x \ne \infty$. See \cite[\S 3-4]{trees} for additional details.

Let $\cD$ be the category whose objects are ordered boron trees, and whose morphisms are embeddings of boron trees that preserve the root and the total order on $T_H$. This is an $\FI$-concrete category via $\vert T_H \vert=T_H$. The category $\cD$ carries a canonical ordering.

\begin{lemma} \label{lem:orderedBoronWQO}
The set $\cM=\cM(\cD)$ is wqo.
\end{lemma}

\begin{proof}
Since trees are finite, this is clearly well founded. If $\cM$ is not well-quasi-ordered, take a bad sequence of isomorphism classes of trees $(T_1,\alpha_1),(T_2,\alpha_2),\dots$ minimal with respect to $|T_i|$.  Since the tree with a single vertex or two vertices embeds into any tree with sufficiently high labels, each $T_i$ has at least three leaves.

Every leaf is therefore connected to a boron atom that has two additional descendants; we refer to them as the left and right descendants of the leaf, using our planar embedding. Define $T_{i,1}$ (respectively, $T_{i,2}$) by taking the root of $T_i$ together with all the hydrogen atoms whose geodesic to the root contains the left (respectively, right) descendant of the root with the induced boron tree structure from $T_i$, and put $\alpha_{i,j} = \alpha_i|_{T_{i,j}}$. By construction both $(T_{i,1},\alpha_{i,1})$ and $(T_{i,2},\alpha_{i,2})$ embed into $(T_i,\alpha_i)$, so by minimality the sub-poset of $\cM$, $\{(T_{i,j},\alpha_{i,j}) \; | \; i \in \bN, 1 \leq j \leq 2\}$, is well-quasi-ordered.

Since each rooted planar boron tree $T_i$ has at least three vertices, $T_{i,1}$ and $T_{i,2}$ are non-empty. This means we can replace $(T_i,\alpha_i)$ in our bad sequence with $((T_{i,1},\alpha_{i,1}),(T_{i,2},\alpha_{i,2}))$. By Dickson's lemma, there is an infinite good subchain,
\[
((T_{i_1,1},\alpha_{i_1,1}),(T_{i_1,2},\alpha_{i_1,2})) \leq ((T_{i_2,1},\alpha_{i_2,1}),(T_{i_2,2},\alpha_{i_2,2})) \leq \cdots
\]
in the component-wise order. Since every leaf in $T_{i,1}$ is smaller in our total order than the leaves in $T_{i,2}$, we find $(T_{i_j},\alpha_{i_j}) \leq (T_{i_k},\alpha_{i_k})$, contradicting badness.
\end{proof}

See \cite{NashWilliams} for details on minimal bad sequences. It is an often used fact---though not always clearly stated---that the subposet of substructures of a minimal bad sequence is well-quasi-ordered. This is the essence of the first paragraph of \cite[Proof of Theorem~1]{NashWilliams}, where $B$ is precisely this subposet. For an application of this theory to Kruskal's tree theorem (with pictures), we refer the reader to \cite[Theorem 1.2]{DraismaNotes}.

\begin{theorem} \label{thm: orderedGrobner}
The category $\cD$ is a G-category.
\end{theorem}

\begin{proof}
We have already said that $\cD$ is orderable, so the result follows from Lemma~\ref{lem:orderedBoronWQO}.
\end{proof}

\begin{corollary}
The category $\cC$ is a QG-category.
\end{corollary}

\begin{proof}
This is witnessed by the forgetful functor $\cD \to \cC$.
\end{proof}

We thus see that $R(\cC)$ is noetherian (Theorem~\ref{thm:noethsys}), and so $S$ is $\Gamma$-noetherian (Theorem~\ref{thm:RtoS}). This completes the proof of Theorem~\ref{thm:boron}.

\begin{remark}
Before we had the general methods of this paper, we attempted to prove Theorem~\ref{thm:boron} by hand, and completely failed. We could not even give a direct proof that $S$ is topologically $\Gamma$-noetherian, which one might expect to be a much easier result. It would be interesting to find a direct argument.
\end{remark}

\end{document}